\documentclass[12pt]{article}

\usepackage{amsmath,amsthm,amsfonts,amscd,amssymb,eucal, bbm, a4wide}
\def\ZZ{{\mathbb Z}}
\def\RR{{\mathbb R}}
\def\CC{{\mathbb C}}
\def\NN{{\mathbb N}}
\def\M{{\mathcal M}}
\def\S{{\mathcal S}}
\def\R{{\mathcal R}}
\def\C{{\mathcal C}}

\newtheorem{theorem}{Theorem}[section]

\begin{document}

\title{\huge On vector configurations
that can be realized 
in the cone of positive matrices} 

\author{{\bf P\'eter E.\ Frenkel\footnote{Supported by FNS and by OTKA grants  K 61116 and NK 72523}}
\\
Department of Mathematics, University of Geneva
\\
2-4 rue du Li\`evre, 1211 Geneve 4, Switzerland
\\
email: {\tt{frenkelp@renyi.hu}} \vspace{1cm}
\\
{\bf Mih\'aly Weiner\footnote{Supported by the ERC Advanced Grant 227458 OACFT ``Operator Algebras
and Conformal Field Theory''}} \\
Department of Mathematics, University of Rome ``Tor
Vergata"
\\
Via della Ricerca Scientifica, I-00133 Roma, Italy
\\
(on leave from: 
Alfr\'ed R\'enyi Institute of Mathematics \\
H-1053 Budapest, POB 127, Hungary) \\
email: {\tt{mweiner@renyi.hu}}
}
\date{}


\maketitle

\begin{abstract} 
Let ${\mathbf v}_1,\ldots ,{\mathbf v}_n$ be $n$ vectors 
in an inner product space. Can we find a $d\in \NN$ and positive 
(semidefinite) matrices $A_1,\ldots,A_n\in M_d(\CC)$ such that ${\rm 
Tr}(A_kA_l)= \langle {\mathbf v}_k,{\mathbf v}_l\rangle$ for all 
$k,l=1,\dots, n$? For  such matrices to exist, one 
must have $\langle {\mathbf v}_k,{\mathbf v}_l\rangle \geq 0$ for all 
$k,l=1,\dots, n$.  We prove that if $n<5$ then this trivial necessary 
condition is also a sufficient one and find an appropriate example showing 
that from $n=5$ this is not so --- even if we allowed realizations by 
positive operators in a von Neumann algebra with a faithful normal tracial 
state. 

The fact that the first such example occurs at $n=5$ is similar to 
what one has in the well-investigated problem of 
{\it positive factorization} of positive 
(semidefinite) matrices. If the matrix $\big(\langle \mathbf v_k,\mathbf v_l\rangle \big)_{(k,l)}$
has a positive factorization, then matrices $A_1$, \dots, $A_n$ as above exist.
However, as we show by a large class of examples constructed with the 
help of the Clifford algebra, the converse implication  is false.
\end{abstract}

\section{Introduction}

\subsection{Motivation}
Throughout this paper, the term ``positive matrix''  will mean ``positive semidefinite matrix''.
The aim of the paper is to study a geometrical property of the cone 
$\C_d$ of positive  matrices in $M_d(\CC)$ in the ``large 
dimensional limit'': we investigate if a given configuration of vectors 
can be embedded in $\C_d$ for {\it some} (possibly very large) $d\in\NN$. 
Note that for $d_1\leq d_2$ we have $\C_{d_1}\hookrightarrow \C_{d_2}$ in a natural 
manner, so if a configuration can be embedded in $\C_{d_1}$, then it can be embedded in $\C_{d_2}$.

To explain the precise meaning of our question, consider the real 
vector space formed by the self-adjoint elements of $M_d(\CC)$. 
It has a natural inner product defined by the formula
\begin{equation}
\langle A, B\rangle \equiv \frac 1d{\rm Tr}(AB) \;\;\;\;\; 
(A^*=A, \, B^*=B\, \in \, M_d(\CC)),
\end{equation}
making it a {\it Euclidean} space. Our cone $\C_d$ is a convex 
cone in this space with a ``sharp end-point'' at 
zero\footnote{Actually one has the much stronger property 
that an element $X$ in this space belongs to the cone $\C_d$
if and only if $\langle X,A\rangle\geq 0$ for all $A\in \C_d$; 
i.e.\! $\C_d$ is a {\it self-dual} cone.}: if 
$A,B\in \C_d$ then $\langle A, B \rangle \geq 0$.

Suppose we are given $n$ vectors $\mathbf v_1,\ldots ,\mathbf v_n$ in a 
Euclidean space. Embedding them in an inner product preserving 
way in $\C_d$ means finding $n$ positive  matrices
$A_1,\ldots A_n \in \C_d$ such that
\begin{equation}
\langle \mathbf v_j, \mathbf v_k \rangle = \langle A_j, A_k\rangle \equiv 
\frac 1d{\rm Tr}(A_j A_k).
\end{equation}
Since, as was mentioned, the angle between any two vectors in $\C_d$ is 
 $\le\pi/2$, one can only hope to embed these vectors 
if $\langle \mathbf v_j,\mathbf v_k\rangle \geq 0$ for all $j,k=1,\ldots 
n$.
So suppose this condition is satisfied. Does it follow that the
given vectors can be embedded in $\C_d$ for some (possibly very large)
$d$? If not, can we characterize the configurations that can be 
embedded? To our knowledge, these questions have not been 
considered in the literature. 

We postpone the summary of our results to the next subsection and note 
that there is  a well-investigated problem --- namely the problem of 
{\it positive factorization} --- which has some relation to 
our questions. The relation between the two topics will  also be
discussed in the next subsection; here we shall only explain our original 
motivation.

If $A$ is positive and $A\neq 0$, then ${\rm Tr}(A)> 0$, so
the affine hyperplane $\{X:\,{\rm Tr}(X)=1\}$ 
intersects each ray of the cone $\C_d$ exactly once and geometric 
properties of this cone can be equally well studied by considering 
just the intersection
\begin{equation}
\S_d = \{A\in M_d(\CC): \, A\geq 0, {\rm Tr}(A)=1\}.
\end{equation}
The compact convex body $\S_d$ is usually referred to as the set of {\it 
density operators}, and it can be naturally identified with the set of 
states of $M_d(\CC)$. Many problems in quantum information theory boil
down to geometrical questions about $\S_d$. From our point of view a
relevant example is the famous open problem about {\it mutually unbiased 
bases}, which turned out \cite{polytope} to be equivalent 
to asking 
whether a certain polytope 
--- similarly to our question --- can be embedded in $\S_d$ or not. Though 
some properties of $\S_d$ have been determined (e.g.\! its volume and surface 
\cite{volume.of.states} is known), its exact ``shape'' is still little
understood.

However, rather than studying the geometry of $\C_d$ (or 
$\S_d$) for a certain $d$, here we are more interested in the ``large 
dimensional'' behaviour. Indeed, often this is what matters; 
think for example of the topic of 
{\it operator monotonuous functions}. (The ordering between self-adjoint 
elements is determined --- in some sense ---  by the geometry 
of $\C_d$, since the operator inequality $A\geq B$ simply means that 
$A-B\in\C_d$.) Indeed, saying that a certain 
function $f:\RR\to\RR$ is monotonuous
(increasing) but {\it not} operator monotonuous means that
though $f(x)\leq f(y)$ for all real $x\leq y$, there exist
two (possibly very large) self-adjoint matrices
$X\leq Y$ such that $f(X)\nleq f(Y)$.

A more direct motivation for our question 
is the {\it Connes 
embedding conjecture}. This conjecture --- in its original form --- is 
about {\it finite} von Neumann algebras; i.e.\! von Neumann algebras 
having a normal faithful tracial state. According to Connes, it should 
be always possible to embed such a von Neumann algebra into an ultrapower 
$\R^\omega$ of the hyperfinite I\!I$_1$-factor $\R$ in a trace-preserving 
way.

At first sight this might seem unrelated to our problem. However, this 
conjecture has many different (but equivalent) forms. For example, it is 
well-known that this embedding property holds if and only if moments of a 
finite set of self-adjoint elements from such a von Neumann algebra $\M$ 
can  always be ``approximated" by those of a set of (complex, self-adjoint) 
matrices. (See e.g.\!  \cite[Prop.\! 3.3]{CollinsDykema} for a proof.) 
Here by ``approximation'' we mean that for every 
$n,m\in\NN, \, \epsilon>0$ and  self-adjoint operators 
$A_1,\ldots A_n\in \M$ it is possible to find a 
$d\in \NN$ and self-adjoint matrices $X_1,\ldots ,X_n\in M_d(\CC)$ such 
that 
\begin{equation}
\big| \tau(A_{j_1}\ldots A_{j_s}) -\tau_d (X_{j_1}\ldots X_{j_s})\big| 
\leq \epsilon
\end{equation} holds for every $s\leq m$ and $j_1,\ldots j_s \in \{1,\ldots,n\}$. Here $\tau$ is a (fixed) faithful, normal, tracial state on $\M$ and 
$\tau_d = \frac{1}{d}{\rm Tr}$ is the normalized trace on $M_d(\CC)$.
Note 
that the Connes embedding conjecture can also be  reformulated in terms of 
linear inequalities between moments \cite{Hadwin}. This approach has 
recently resulted $\cite{Florin,KlepSchweighofer}$ in new forms of the 
conjecture which are in some sense similar to Hilbert's $17^{\rm th}$ 
problem and are formulated {\it entirely} in terms of moments of matrices 
(i.e. in a way which apparently does not involve von Neumann 
algebras other than those of finite matrices).

It is then natural to ask: what can we say about moments of self-adjoint 
matrices, in general? Of course, there is not too much to say about the 
first and second moments. The set of numbers $\{{\rm 
Tr}(X_j):\,j=1,\ldots, n\}$ can be any subset of $\RR$ and the only 
condition on the second moments is that the matrix $\big( {\rm Tr}(X_jX_k) 
{\big)}_{(j,k)}$ must have only real entries and must be positive, as it is  the {\it Gram matrix}
of $X_1,\ldots, X_n$ when these are viewed as vectors in a 
Euclidean space.

While the first two moments are too banal to be interesting, higher 
moments are too complicated to be fully understood. In this respect it 
seems a good way ``in between'' to study moments of the form ${\rm 
Tr}(X_j^2 X_k^2)$. Though they are $4$-moments, they can be also 
considered as $2$-moments of the {\it positive} matrices $Y_j:=X_j^2\; 
(j=1,\ldots, n)$, which again leads to our question.

\subsection{Relation to positive factorization
and main results}

There is a certain type of configuration which can be  embedded in 
the cone of positives in a trivial manner. Indeed, let $\mathbf v_1,\ldots, 
\mathbf v_n$ be vectors from the {\it positive orthant} of $\RR^d$; i.e.\! 
vectors of $\RR^d$ with only nonnegative entries (in which case of course 
$\langle \mathbf v_k,\mathbf v_l\rangle \geq 0$ follows automatically for 
all $k,l=1,\ldots, n$). Then for each $k=1,\ldots, n$ let $A_k$ be $\sqrt d$ times the 
$d\times d$ diagonal matrix whose diagonal entries are simply the entries 
of $\mathbf v_k\in \RR^d$ (listed in the same order). It is now trivial 
that the matrices $A_k$ are positive and that
$\langle A_k,A_l\rangle \equiv \frac 1d {\rm Tr}(A_kA_l) =\langle \mathbf 
v_k,\mathbf v_l\rangle$ for all $k,l=1,\ldots , n$. 

Note that in the above realization all matrices commute. Actually a 
certain converse of this remark is also true: if a vector configuration 
can be realized by commuting elements of $\C_d$, then this configuration 
can be embedded in the positive orthant  $\RR^d_{\ge 0}$. This simply follows 
from the fact that a set of positive, commuting matrices can be 
simultaneously diagonalized and that all diagonal entries of positive 
matrices must be nonnegative.

No more than $4$ vectors span at most a $4$-dimensional space and 
in \cite{gw} it is proved that if 
$\mathbf v_1,\ldots,\mathbf v_4 \in \RR^4$ are such that 
$\langle \mathbf v_k,\mathbf v_l\rangle \geq 0$ for all 
$k,l=1,\ldots ,4$ then there exists an orthogonal transformation 
$O:\RR^4\to \RR^4$ such that $O\mathbf v_1$, \dots, $O\mathbf v_4$
 all lie in the positive orthant of $\RR^4$. It follows 
at once that any configuration of $n\le 4$ vectors such that the angle between 
any two is  $\le\pi/2$ can be embedded in the
cone $\C_4$.

Another elementary observation is that {\it any} number of vectors 
$\mathbf v_1,\ldots \mathbf v_n $ spanning a no more than  
$2$-dimensional space and satisfying $\langle \mathbf v_k,\mathbf 
v_l\rangle \geq 0$ for all
$k,l=1,\ldots ,n$ can be embedded in the positive orthant of $\RR^2$
and hence also in the cone $\C_2$. Thus a configuration, with all inner products nonnegative, which does not 
have a realization in $\C_d$ for any $d\in \NN$ must consist of at least 
$5$ vectors and  must span a space of dimension $\ge 3$. 
In fact, as we shall show in the last section, there exists a configuration of $5$ vectors in 
$\RR^3$ with all pairwise inner products nonnegative
but with no realization in any 
cone $\C_d\; (d\in \NN)$. In connection to the Connes embedding 
conjecture it is interesting to note that --- as we shall explain ---
this particular configuration  cannot be embedded in the cone of 
positives of any finite von Neumann algebra.

An $n\times n$ matrix $A$ is said to have a positive factorization iff 
there exists another (possibly non-square) matrix $B$ such that all 
entries of $B$ are nonnegative reals and $A=B^T B$. A trivial necessary 
condition for the existence of a positive factorization is that $A$ must 
be positive and all entries of $A$ must be nonnegative reals. 

Now suppose $\mathbf v_1,\ldots , \mathbf v_n$ are vectors satisfying 
$\langle \mathbf v_k,\mathbf v_l\rangle \geq 0$, and consider their
Gram matrix $A$; that is, the $n\times n$ 
matrix whose $k,l$-th entry is $\langle \mathbf v_k,\mathbf 
v_l\rangle$. Then $A$ is  positive and has only nonnegative reals as
its entries. So assume $A$ has a positive factorization: $A=B^T B$ for 
some $m\times n$ matrix $B$ with only nonnegative real entries. Then a 
trivial check shows that the 
map
\begin{equation}
\mathbf v_k \mapsto {\rm the} \; k^{\rm th}\; {\textrm{column vector of }} 
B
\end{equation}
is an (inner product preserving) embedding of our vector configuration into 
the positive orthant of $\RR^m$, and hence that the configuration can be 
realized in $\C_m$. By the same argument it is also clear that the Gram 
matrix of a vector configuration has a positive factorization if and only 
if the given configuration can be embedded in the positive orthant of 
$\RR^m$ for some $m\in \NN$.
Note that it is long known \cite{posfact1,posfact2,gw, hl, bb} that an
$n\times n$ matrix $A$
with only nonnegative entries always has a 
positive factorization if $n<5$, and that 
for $n\ge 5$, as counterexamples show, the same 
 assertion fails.

However, even if the Gram matrix of a given vector configuration
does not have a positive factorization, the
 configuration might still be embeddable  into $\C_d$. In fact, in the next
section we shall give a construction showing that if there exists a vector
$\mathbf w$ such that the angle between $\mathbf w$ and $\mathbf v_k$ is
$\le \pi/4$ for all $k=1,\ldots, n$, then the 
configuration $\mathbf v_1,\ldots, \mathbf v_n$
can be embedded in $\C_{d}$ where $d=2^{\lfloor n/2\rfloor}$.
In general though, as we shall prove, such a configuration cannot be
realized in a positive orthant. By an  earlier remark this  implies
that in general such a configuration --- though it can be realized by
positive matrices --- cannot be realized by {\it commuting} positive
matrices. In fact, our embedding construction relies on the Clifford algebra, which is non-commutative.

Our results, in some sense, can be considered as first examples. Finding a 
suitable characterization of the configurations that can be realized 
by positive matrices remains an open problem.

\section{Embeddings via the Clifford algebra}
In this section we assume (without loss of generality) that the  vectors to be represented by positive matrices are given in $\RR^n$.
We prove representability if the vectors are contained in the spherical cone
\[C_n=\{\mathbf v=(x_1,\dots, x_n)\in\RR^n\mid x_1^2\ge x_2^2+\dots+x_n^2\}\]
formed by vectors whose angle with the vector $(1,0,\dots, 0)$ does not exceed $\pi/4$.

\begin{theorem}\label{cliff}
Let $d=2^{\lfloor n/2\rfloor}$. There exists an isometric real linear embedding $\phi$ of $\RR^n$ into the space of $d\times d$ self-adjoint
complex matrices that maps $C_n$ into the cone $\mathcal C_d$ of positive matrices.
\end{theorem}

\begin{proof}
First assume that $n=2k+1$ is odd. Then $d=2^k$.  We identify \[\RR^n=\RR\oplus \RR^k\oplus \RR^k\]
and
\[M_d(\CC)= {\mathrm {End}}_\CC\left(\bigwedge \CC^k\right)=\CC\otimes {\mathrm {End}}_\RR\left(\bigwedge \RR^k\right),\]
where the space of anti-symmetric tensors is endowed with the inner product
\[\left\langle\bigwedge\mathbf v_i,\bigwedge\mathbf w_j\right\rangle=\det\left(\langle\mathbf v_i,\mathbf w_j\rangle\right).\]
When $\mathbf v\in\RR^k$, we write $\epsilon_{\mathbf v}\in{\mathrm {End}}\left(\bigwedge \RR^k\right)$ for the map $\mathbf u\mapsto \mathbf v\wedge \mathbf u$, where $\mathbf u$ is any anti-symmetric tensor. 
We use the anticommutator notation $\{a,b\}=ab+ba$. It is well known that
\[{\rm Tr}\; \epsilon_{\mathbf v}=
{\rm Tr} \;\epsilon^*_{\mathbf v}=0,\]  \[ \{\epsilon_{\mathbf v},\epsilon_{\mathbf w}\}=
\{\epsilon^*_{\mathbf v},\epsilon^*_{\mathbf w}\}=0\]
and
\[\{\epsilon_{\mathbf v},\epsilon^*_{\mathbf w}\}=\langle \mathbf v,\mathbf w\rangle I.\]
We define
\[\phi: c\oplus \mathbf v\oplus\mathbf  w\mapsto cI+\epsilon_{\mathbf v}+\epsilon^*_{\mathbf v}+\sqrt{-1}(\epsilon_{\mathbf w}-\epsilon^*_{\mathbf w}).\]
This clearly maps $\RR^n$ to self-adjoint matrices in a linear way.
Using the above anticommutation relations, we have
\[\phi(c\oplus\mathbf v\oplus\mathbf w)^2=c^2 I+\langle \mathbf v,\mathbf v\rangle I
+\langle \mathbf w,\mathbf w\rangle I+2c\left(\epsilon_{\mathbf v}+\epsilon^*_{\mathbf v}+\sqrt{-1}(\epsilon_{\mathbf w}-\epsilon^*_{\mathbf w})\right).\]
We deduce that $\phi$ is an isometry since
\begin{align*}
|\phi(c\oplus\mathbf v\oplus\mathbf w)|^2=\frac 1d{\rm Tr}\left(\phi(c\oplus\mathbf v\oplus\mathbf w)^2\right)=\\=\frac
1d{\rm Tr}\left(c^2 I+\langle \mathbf v,\mathbf v\rangle I
+\langle \mathbf w,\mathbf w\rangle I\right)=|
c\oplus\mathbf v\oplus\mathbf w|^2.\end{align*}
Now assume that 
$c\oplus\mathbf v\oplus\mathbf w\in C_n$, i.e.,  $c^2\ge|\mathbf v|^2+
|\mathbf w|^2$.
We have
\[\phi(c\oplus\mathbf v\oplus\mathbf w)=cI+\phi(0\oplus\mathbf v\oplus\mathbf w).\]
Here the last term is a self-adjoint matrix that squares to
$\left(|\mathbf v|^2+|\mathbf w|^2\right)I$, so its operator norm equals its Euclidean norm, both being the common absolute value of all its eigenvalues, namely
$\sqrt{|\mathbf v|^2+|\mathbf w|^2}\le c$.
Thus, $\phi(c\oplus\mathbf v\oplus\mathbf w)$ is  a positive matrix, which finishes the proof for $n$ odd.

Now assume that $n=2k$ is even.  We still have $d=2^k$, and there is an obvious isometric embedding $\RR^n\to\RR^n\oplus \RR=\RR^{n+1}$ that maps $C_n$ into $C_{n+1}$, so the problem is reduced to the odd-dimensional case.
\end{proof}

As a contrast,
we shall now construct a configuration of six vectors in the circular cone $C_3\subset\RR^3$ that cannot be isometrically embedded
into the positive orthant $\RR^d_{\ge 0}$ for any $d$.
 
Put
\[\mathbf v_k=\left(1,\cos\frac {2\pi k}6,\sin \frac {2\pi k}6\right)\in C_3\subset\RR^3\qquad (k\in\ZZ/6\ZZ).\]

\begin{theorem}\label{hexa}
It is impossible to have positive vectors $\mathbf a_0,\dots, \mathbf a_5\in\RR^d_{\ge 0}$ with \begin{equation}\label{isom6}
\langle\mathbf a_j,\mathbf a_k\rangle=\langle\mathbf v_j,\mathbf v_k\rangle {\textrm{ for all }} j,k\in\ZZ/6\ZZ.\end{equation}
\end{theorem}

\begin{proof}
By contradiction, suppose that the positive vectors $\mathbf a_0,\dots, \mathbf a_5\in\RR^d_{\ge 0}$  satisfy \eqref{isom6}.
Put
\[\mathbf e=\frac{\mathbf a_k+\mathbf a_{k+3}}2;\] this is independent of $k$ because the analogous statement holds for the $\mathbf v_k$.
Put \[\mathbf b_k=\mathbf a_k-\mathbf e=\mathbf e-\mathbf a_{k+3}.\]
Then $|\mathbf b_k|=|\mathbf e|$ and $\mathbf e\pm\mathbf b_k\in\RR^d_{\ge 0}$, so $\mathbf b_k=\mathbf s_k\circ \mathbf e$ (coordinatewise product) for a suitable vector of signs $\mathbf s_k\in\{-1,1\}^d$.  But $\mathbf b_0+\mathbf b_2+\mathbf b_4=\mathbf 0$, so
$(\mathbf s_0+\mathbf s_2+\mathbf s_4)\circ \mathbf e=\mathbf 0$,
which is impossible because every coordinate of $\mathbf s_0+\mathbf s_2+\mathbf s_4$ is an odd integer and $\mathbf e\ne\mathbf 0$.
\end{proof}

\section{A configuration which cannot be realized}
Put
\[\mathbf v_k=\left(\sqrt{\cos\frac\pi 5},\cos\frac {2\pi k}5,\sin \frac {2\pi k}5\right)\in\RR^3\qquad (k\in\ZZ/5\ZZ).\]
Observe that
\[\langle \mathbf v_k,\mathbf v_{k+1}\rangle=\cos(\pi/5)+\cos(2\pi/5)>0\]
and
\[\langle \mathbf v_k,\mathbf v_{k+2}\rangle=\cos(\pi/5)+\cos(4\pi/5)=0,\]
so $\langle\mathbf v_j,\mathbf v_k\rangle\ge 0$ for all $j,k\in \ZZ/5\ZZ$.

Let $\M$ be a  von Neumann algebra with a  fixed 
faithful trace $\tau:\M\to\CC$. 
(The simplest example is $\M=M_d(\CC)$ with $\tau$ being (a constant multiple of) the ordinary trace.)
\begin{theorem}\label{penta}
It is impossible to have positive elements $A_0,\dots, A_4\in\M$ with \begin{equation}\label{isom}\tau(A_jA_k)=\langle\mathbf v_j,\mathbf v_k\rangle {\textrm{ for all }} j,k\in\ZZ/5\ZZ.\end{equation}
\end{theorem}

\begin{proof}
By contradiction, suppose that the positive elements $A_0,\dots, A_4\in\M$ satisfy \eqref{isom}.
Then
\[\tau(\sqrt{A_k}A_{k\pm 2}\sqrt{A_k})=\tau(A_kA_{k\pm 2})=\langle\mathbf v_k,\mathbf v_{k\pm 2}\rangle=0,\]
but
\[\sqrt{A_k}A_{k\pm 2}\sqrt{A_k}=\sqrt{A_k}\sqrt{A_{k\pm 2}}\left(\sqrt{A_k}\sqrt{A_{k\pm 2}}\right)^*\]
is positive, so it is zero (since $\tau$ is faithful).
Thus, $\sqrt{A_k}\sqrt{A_{k\pm 2}}=0$ and so
\[{A_k}{A_{k\pm 2}}=\sqrt{A_k}\sqrt{A_k}\sqrt{A_{k\pm 2}}\sqrt{A_{k\pm 2}}=0.\]

Now observe that $\mathbf v_k$, $\mathbf v_{k+2}$ and  $\mathbf v_{k-2}$ 
form a basis of $\RR^3$, in particular, we have
\[\mathbf v_{k\pm 1}\in\RR \mathbf v_k+\RR \mathbf v_{k+2}+\RR\mathbf v_{k-2}.\]
Since the mapping $v_k\mapsto A_k$ preserves inner products, it follows that
\[ A_{k\pm 1}\in\RR A_k+\RR A_{k+2}+\RR A_{k-2}.\]
Multiplying on either side by $A_k$, we get that \[0\ne A_kA_{k\pm 1}=A_{k\pm 1}A_k\in \RR A_k^2.\]
Note that the product here is nonzero because its trace is strictly positive. We deduce
\[\RR A_0^2=\RR A_1A_0=\RR A_1^2,\]
but the $A_k$ are positive, and the positive square root of a  positive operator is unique, so this implies $\RR A_0=\RR A_1$. 
We have an isometry $\mathbf v_k\mapsto A_k$, so then
$\RR\mathbf v_0=\RR \mathbf v_1$, a contradiction.
\end{proof}

\bigskip

\noindent
{\bf Acknowledgement.} We thank Florin R\u{a}dulescu for the suggestion 
of the problem.

\end{document}